\documentclass{amsart}
\usepackage{color}

\newtheorem{Theorem}{Theorem}[section]
\newtheorem{Lemma}[Theorem]{Lemma}

\newcommand{\G}{\mathcal{G}}

\begin{document}
\date{\today}

\title{Distinguishing actions of symmetric groups and related graphs}

\author{Mariusz Grech, Andrzej Kisielewicz}
\address{Faculty of Pure and Applied Mathematics, Wroc\l aw University of Science and Technology \\
Wybrzeże Wyspiańskiego Str. 27,
50-370 Wrocław, Poland}
\email{[mariusz.grech,andrzej.kisielewicz]@pwr.edu.pl}
\thanks{
{Supported in part by Polish NCN grant 2016/21/B/ST1/03079}}


\maketitle

\begin{abstract} 
The distinguishing number $D(G,X)$ of an action of a group $G$ on a set $X$ is the least size of a partition of $X$ such that no element of $G$ acting nontrivially on $X$ preserves this partition.
In this paper we describe the distinguishing numbers for all actions of the symmetric group  $S_n$, for any $n\geq 3$. This allows us to describe the distinguishing numbers for all graphs whose automorphism group is isomorphic with a symmetric group. Our description solves a few open problems posed by various authors in earlier papers on this topic.  
\end{abstract}

\section{Introduction} 
Let $G$ be a group acting on a set $X$. The \emph{distinguishing number} $D(G,X)$ of this action is the least number $d$ such that there exists a partition of $X$ into $d$ blocks with the property that the only elements in $G$ preserving this partition are those fixing all $x\in X$. The partition involved is often considered as a $d$-\emph{labeling} of $X$. Without loss of generality we may restrict to faithful actions of $G$ (cf. \cite{chan}). This means that we may restrict to the study of the distinguishing number $D(G,X)$ of permutation groups $G \leq Sym(X)$.

If $G=Aut(\G)$ is an automorphism group  of a graph $\G$, then $D(G,V(\G))$ is called the \emph{distinguishing number of} $\G$ and denoted $D(\G)$. Then it may be defined in terms of \emph{colorings} of graphs breaking the symmetry of a graph. In fact, this is how the distinguishing number has been originally introduced by Albertson and Collins in \cite{albe}. They also defined the \emph{distinguishing set} $D(G)$ for an abstract group $G$ as the set of distinguishing numbers of $D(\G)$ with $Aut(\G) \cong G$.

It was Tymoczko in \cite{tymo} who extended the concept for group actions (cf. also \cite{chan}). This turned out to be very fruitful when, in 2011 \cite{bail}, it was realized that in permutation group theory the problem had been investigated for many years as a part of the study of regular sets.

The investigation of the distinguishing number of the actions of the symmetric group $S_n$ has been initiated already in \cite{albe}, where it was proved, in Theorem~7, that the set $D(S_4) =\{2,4\}$. The authors made a conjecture that 
if $n\geq 4$, then $n-1 \notin D(S_n)$ (Conjecture~2).

In \cite[Theorem 4.2]{tymo} Tymoczko, verifying another conjecture in \cite{albe}, has proved that if $Aut(\G)=S_n$ for a graph $\G$ and $D(\G)=n$, then one orbit of $\G$ is a copy of $K_n$ or
$\overline K_n$ and the rest are $1$-orbits. 
She also asked if it is true that for arbitrarily large $n$ there exist faithful group actions of $S_n$ on $X$ with $D(S_n,X) = n-1$ (Question~3.2).

All this study involves analyzing action of groups on different orbits. 
In \cite[Proposition 1.2]{lieb} Liebeck has observed that if  $Aut(\G) \cong S_n$  for a graph $\G$, $n>6$, and all orbits of $Aut(\G)$  have size less then 
$\binom{n}{2}$, then all these orbits have size $n$ or $1$. Using this, 
in \cite[Theorem 4.2]{klav}, Klav\v{z}ar, Wong, and Zhu, have proved that for such graphs $D(G,V) =
\lceil \sqrt[k]{n}\rceil$, for some positive integer $k$.
They also proved some additional results supporting the conjecture that the assumption on the size of orbits can be removed (Conjecture 4.5).

In this paper we generalize all the mentioned results, verify Conjecture~3, Conjecture 4.5, and answer Question~3.2.  
Our main result is as follows:

\begin{Theorem} \label{main}
Let $(S_n,X)$ denote a faithful action of $S_n$  on a set~$X$ for some  $n\geq 3$. Then, either $D(S_n,X)=2$ or one of the following holds:
\begin{enumerate}
    \item  Each orbit has size $1, 2, n$ or $2n$,  and with 
$k$ and $r$ denoting the numbers of the orbits of size $n$ and $2n$, respectively,
 $$D(S_n,X)=\left\{
 \begin{array}{ll}
\left\lceil  \sqrt[k]{n}\; \right\rceil, & \textrm{if there is no orbit of size $2$ or $2n$},  \\
\lceil  \sqrt[k+2r]{n-1}\; \rceil, & \textrm{otherwise;}
 \end{array}\right.$$
\item $(S_n,X)$ corresponds to one of the exceptional permutation groups listed in Table~$1$.
\end{enumerate}
\end{Theorem}

The table is given in Section~\ref{s:re}, where we also give a more detailed version of this result describing 
the structure of the corresponding permutation groups with the distinguishing number greater than $2$. 
Applying this to graphs we get that, essentially, there is only one exception. This is Petersen graph $\mathcal P$ on $10$ vertices with $Aut(\mathcal P)\cong S_5$ and $D(\mathcal P)=3$. Obviously, the complement of $\mathcal P$ has the same parameters. Moreover, if $\mathcal{P}$ (or its complement) is extended with some fixed points so that the automorphism group remains the same (in particular, each fixed point is joined to all or no vertices of $\mathcal{P}$), then the resulted graph  has the same  distinguishing number $D=3$, as well. The general result is the following.

\begin{Theorem} \label{main2}
Let $\G$ be a graph with $Aut(\G)\cong  S_n$. Then either $D(\G)=2$ or one of the following holds:
\begin{enumerate}
\item  $Aut(\G)$ has only orbits of size $n$ or $1$, and  $D(\G)=\left\lceil  \sqrt[k]{n} \; \right\rceil$, where $k$  is the number of the orbits of size $n$;
\item $\G$ is the Petersen graph or its complement, extended  possibly with some fixed points as described above.
\end{enumerate}
\end{Theorem}

Theorem~\ref{main2} above extends \cite[Theorem~4.2]{klav} and verifies positively  \cite[Conjecture~4.5]{klav}.  In turn, \cite[Question 3.2]{tymo} is settled positively by Theorem~\ref{main}. Indeed, it is enough to take a faithful action of $S_n$ with one orbit of size $n$, one orbit of size $2$ and no orbit of size $2n$. We show in Section~\ref{s:re} that there are such actions for every $n\geq 3$. Then, by Theorem~\ref{main}, for $n>3$, we have $D(S_n,X) = n-1$. This is not the case for automorphism groups of graphs: here, if $n > 3$, then $n-1 \notin D(S_n)$, which settles   \cite[Conjecture~2]{albe} posed 25 years ago.

Liebeck \cite{lieb} has described the structure of the graphs $\G$ with $Aut(\G)\cong S_n$ and orbits only of size $n$ or $1$. Each such graph must be a disjoint union of $k\geq 1$ components, each of which is a clique or an anti-clique of size $n$, and an arbitrary number $t\geq 0$ of fixed points. There are only few types of connections possible between these components, and some additional conditions have to hold (we refer the reader to \cite{lieb} for details). Anyway, Petersen graph in this context has a really exceptional structure.

The rest of the paper is organized as follows. In Section~\ref{s:ss} we recall the concept of a subdirect sum of permutation groups and the structure theorem on intransitive permutation groups, which is the main tool allowing us to solve all mentioned problems. In Section~\ref{s:fa} we establish some more general facts on 
faithful actions of groups we use later in our proof. Section~\ref{s:pr} contains the proof of Theorem~\ref{main}  for $n>6$. For $n\leq 6$ exceptions arise, and we deal with this small cases in Section~\ref{s:re}, with the help of a computational system.

For more detailed information on other research concerning distinguishing numbers the reader is referred to \cite{gksim}.

\section{Subdirect sum of permutation groups} \label{s:ss}

As we have mentioned, in the study of the distinguishing number we face a simultaneous action of the automorphism group of a graph (or an abstract group) on various orbits, and we need to know how the actions on various orbits are interrelated.
In this section we recall the \emph{subdirect sum of permutation groups} and the general structure theorem for intransitive permutation groups. An additional aim of this paper is to show that this special knowledge allows us to approach various problems that remained unsolved so far.

All permutation groups we consider are finite and considered up to \emph{permutation isomorphism} \cite[p.~17]{DM} (i.e., two permutation groups that differ only in labeling of points are treated as equal).

Given two permutation groups $G \leq Sym(X)$ and $H \leq Sym(Y)$, the \emph{direct sum} $G\oplus H$ is the permutation group on the \emph{disjoint}
union $X \cup Y$ defined as the set of all permutations $(g,h)$, $g\in G, h\in H$ such that (writing permutations on the right)
$$
x(g,h) =
\left\{\begin{array}{ll}
xg, & \mbox{if } x\in X\\
xh, & \mbox{if } x\in Y
\end{array}\right.
$$
Thus, in $G\oplus H$, permutations of $G$ and $H$ act independently in a natural way on the disjoint union of the underlying sets.
 
We define the notion of the \emph{subdirect sum} following \cite{GK1} (and the notion of \emph{intransitive product} in \cite{kis1}).
Let $H_1\unlhd\; G_1 \leq S_n$ and $H_2\unlhd\; G_2\leq S_m$ be permutation groups such that $H_1$ and $H_2$ are normal subgroups of $G_1$ and $G_2$, respectively. Suppose, in addition, that factor groups $G_1/H_1$ and $G_2/H_2$ are (abstractly) isomorphic and $\phi : G_1/H_1 \to G_2/H_2$ is the isomorphism mapping. Then, by
$$
G = G_1[H_1] \oplus_\phi G_2[H_2]
$$
we denote the subgroup of $G_1 \oplus G_2$ consisting of all permutations $(g,h)$, $g\in G_1, h\in G_2$, such that $\phi(H_1g) = H_2h$.
Each such group will be called the \emph{subdirect sum} of $G_1$ and $G_2$. It is easily checked that such a set of permutations forms a permutation group.
 
If $H_1=G_1$ and $H_2=G_2$, then $G = G_1 \oplus G_2$ is the usual direct sum of $G_1$ and $G_2$.
If $H_1$ and $H_2$ are trivial one-element subgroups, then $\phi$ is an isomorphism of $G_1$ onto $G_2$, and the sum is called, in such a case, the \emph{parallel sum} of $G_1$ and $G_2$. Then the elements of $G$ are of the form $(g,\phi(g))$, $g\in G_1$, and both the groups act in a parallel manner on their sets \emph{via} isomorphism $\phi$. In general, the trivial permutation group on $n$ points is denoted $I_n$. But in this case we omit trivial subgroups in brackets and use the notation $G=G_1\oplus_\phi G_2 = G_1[I_n]\oplus_\phi G_2[I_m]$, where $\phi$ is an (abstract) isomorphism between $G_1$ and $G_2$. (In other papers, also the notation $G=G_1|| G_2$, whenever there is no need to refer to $\phi$).
We emphasize that $G_1$ and $G_2$ need to be abstractly isomorphic, but not necessarily permutation isomorphic, and they may act on sets of different cardinalities with $n\neq m$.
 
In the special case when, in addition, $G_1=G_2=G$ and $\phi$ is the identity, we write $G^{(2)}$ for $G\oplus_\phi G$. More generally, for $r\geq 2$, by $G^{(r)}$ we denote the permutation group in which the group $G$ acts in the parallel way (\emph{via} the identity isomorphisms) on $r$ disjoint copies of a set $X$. This group is called
the \emph{parallel multiple} of $G$. In particular, we admit $r=1$ and put $G^{(1)}=G$.
For example, the cyclic group generated by the permutation
$g = (1,2,3)(4,5,6)(7,8,9)$ is permutation isomorphic to the parallel multiple $C_3^{(3)}$, where $C_3$ denotes the permutation group on $\{1,2,3\}$ generated by the cycle $(1,2,3)$.
 
The main fact established in \cite{kis1} is that every intransitive group has the form of a subdirect sum, and its components can be easily described. Let $G$ be an intransitive group acting on a set $X = X_1 \cup X_2$ in such a way that $X_1$ and $X_2$ are disjoint fixed blocks of $G$. Let $G_1$ and $G_2$ be restrictions of $G$ to the sets $X_1$ and $X_2$, respectively (they are called \emph{constituents}). Let $H_1 \leq G_1$ and $H_2 \leq G_2$ be the subgroups fixing pointwise $X_2$ and $X_1$, respectively. Then we have
 
\begin{Theorem} \cite[Theorem 4.1]{kis1} \label{intr}
If $G$ is a permutation group as described above, then
$H_1$ and $H_2$ are normal subgroups of $G_1$ and $G_2$, respectively,
the factor groups $G_1/H_1$ and $G_2/H_2$ are abstractly isomorphic, and
$$G = G_1[H_1] \oplus_\phi G_2[H_2],$$
where $\phi$ is an isomorphism of the factor groups.
\end{Theorem}

Note that the operator $\oplus_\phi$ is commutative in the sense that $X$ may be written as $X = X_2 \cup X_1$, as well, and $G = G_2[H_2] \oplus_{\phi^{-1}} G_1[H_1]$. In case when $H_1$ or $H_2$ are trivial we omit them in notation, as for the parallel sum above.

There is some subtlety regarding the parallel powers we have to be aware. A well-known fact is that the automorphism group $Aut(G)$ of a group $G$ may have outer automorphisms that are not given by the conjugation action of an element of $G$. In the case of permutation groups $G\leq Sym(X)$ some outer automorphism may still be given by the conjugation action of an element in $Sym(X)$. This corresponds generally to permuting elements of $X$, and such automorphisms are called \emph{permutation automorphisms}. They form a subgroup of $Aut(G)$, which we denote by $PAut(G)$. Often $PAut(G)=Aut(G)$, but some permutation groups have also other automorphisms, which we will call \emph{nonpermutation} automorphisms.
 
Now, if $G = H\oplus_\psi H$, for some permutation group $H$, where $\psi\in PAut(H)$, then $G$ is permutation isomorphic to $H^{(2)}$ (i.e., $G=H^{(2)}$, according to our convention). If $\psi$ is a nonpermutation automorphism, then $G \neq H^{(2)}$. (More precisely, we should speak here about isomorphisms induced by automorphisms and make distinction between base sets of components, but we assume that this is contained in the notion of the \emph{disjoint union}, and we will make it explicit only when the need arises). An example is the symmetric group $S_6$ that has a nonpermutation automorphism $\psi$ (cf. \cite{CL}). Then, $S_6^{(2)}$ and $S_6 \oplus_\psi S_6$ are not permutation isomorphic.

This is an exception, since it is known that for $n\neq 6$, $S_n$ has no nonpermutation automorphism, and for each automorphism $\phi$,  $S_n \oplus_\phi S_n = S_n^{(2)}$.

\section{Faithful actions of groups}\label{s:fa}

Now, using the results of the previous section, we establish some simple fact concerning faithful actions of groups, and in particular, actions of the symmetric group $S_n$.

Let $H$ be a subgroup of a group $G$. By $G:H$ we denote the right coset space $G:H = \{Hg : g\in G\}$. 
It is a well-known fact that if a group $G$ acts on a set $X$ transitively, then its action is equivalent to the action of $G$ on the coset space $G:H$ for some  $H\leq G$. This action is faithful if and only if $H$ contains no nontrivial normal subgroup of $G$.

We wish to characterize intransitive actions. 
Faithful actions may be identified with the corresponding permutation groups. To proceed with $S_n$ we need first two simple general facts.

In general, an intransitive permutation group with $k$ orbits may be written symbolically as 
$$G =G_1[H_1] \oplus_{\phi_1} \ldots \oplus_{\phi_{k-1}} G_k[H_k].$$ 
As we have noted, the order of the summands may be chosen here arbitrarily. Also, to make the notation formal, and to establish the order of operations, we need to put suitable brackets, and this may be also done to put arbitrary order. Changing this order changes usually actual isomorphisms $\phi_1,\ldots, \phi_{k-1}$, and the normal subgroups $H_1,\ldots,H_k$ in the notation, but the constituents $G_1,\ldots,G_k$ remain, in any case, unchanged.

It is well known that if group  $G$ acts on a set $X$, and $Y$ is an orbit in this action, then the action of $G$ on $Y$ is equivalent to a faithful action of the factor group $G/H$ on $Y$ for some normal subgroup $H\unlhd G$ ($H$ is actually the kernel of the action). We also have the following converse.

\begin{Lemma} \label{f:a}
Let $G$ be a group and $H$ its normal subgroup. Let $(G/H, Y)$ be a faithful action of $G/H$ on a set $Y$.  Then there exists a faithful intransitive action of $G$ on a set $X\supseteq Y$ such that $Y$ is an orbit in this action and the action of $G$ on $Y$ is equivalent to $(G/H, Y)$.  
\end{Lemma}

\begin{proof}
Let $N$ be an arbitrary transitive permutation group isomorphic to $G$, an $N'$ the permutation group corresponding to $(G/H,Y)$.
Consider the permutation group $$M=N[H] \oplus_\phi N' \; (\; = N[H] \oplus_\phi N'[I_m]).$$ 
As $N/H \cong N'$, the definition is correct. Moreover, for every permutation $g\in N$, there is precisely one permutation $h\in N'$ such that $(g,h)\in M$. Consequently, $M\cong G$. As the second constituent $N'\cong G/H$, the result follows.  
\end{proof}

The permutation group $M$ constructed in the proof above has the property that one of the constituents $N$ is isomorphic to the group itself. Then every permutation on the orbit corresponding to $N$ determines uniquely permutations on other orbits. In terms of the action of a group $G$ on a set $X$ this means that if the action of $G$ on one of the orbits is faithful then the action on this orbits determines uniquely the action on other orbits. 
This leads to the following corollary on distinguishing numbers.

\begin{Lemma}\label{f:d}
Let a group $G$ act on a set $X$ and let $Y\subseteq X$ be a set preserved in this action. If the action of $G$ on $Y$ is faithful, then $D(G,X)\leq D(G,Y)$.
\end{Lemma}

Now, we apply the above consideration to the case $G=S_n$. For $n\neq 4$, $S_n$ has one nontrivial normal subgroup $H=A_n$, and two trivial ones: $H=Id$ and $H=S_n$. So any constituent of $S_n$ has to be isomorphic to: $S_n$, $S_2$ or $S_1$.
Since all transitive actions of $S_n$ are faithful except for action equivalent to that on $S_n : A_n$ or $S_n : S_n$, it follows that
either the orbit is of size $\geq n$, and the action of $S_n$ on such an orbit is faithful or the orbit is of size $2$ or $1$, and the action of $S_n$ on these orbits is equivalent to the action of $S_2$ or $S_1$. The latter are just fixed points of the action. It follows that if $(S_n, X)$ is a faithful action of $S_n$ on $X$, then there is at least one orbit $X_i$ in this action of size at least $n$.

In the case of $G=S_4$ there is one more normal subgroup $K_4$. We have $S_4/K_4 \cong S_3$. So, we have in addition two transitive actions of $S_3$ on $3$ and on $6$ points. If $S_4$ is a faithful action on $X$, then since  $K_4 < A_4$, there must be at least one orbit with the constituent $S_4$. 
Thus we have

\begin{Lemma} \label{f:f}
If $S_n$ acts faithfully 
on a set $X$ and $n\geq 3$, then there is an orbit $Y$ of size $\geq n$ in this action on which $S_n$ acts faithfully.  
\end{Lemma}

Note that if $S_n$ acts transitively on a set of $n$ points, then the only possible action up to equivalence is $(S_n,S_n:S_{n-1})$, which is the natural action of $S_n$. So, if all the orbits of a permutation group $G$ abstractly isomorphic to $S_n$  are of size  $n$, then $G$ acts faithfully on each orbit, and the action on one orbit determines the action on other orbits.  Consequently, $G = S_n^{(k)}$ is a parallel power (with exception of $n=6$, when some of components may be related by a nonpermutation automorphism; see the end of Section~\ref{s:ss}). This result may be extended for the case with fixed points.

\begin{Lemma}\label{f:n}
If $G$ is a  permutation group abstractly isomorphic to $S_n$, $n\neq 6$, and all orbits of $G$ are of size $n$ or $1$, then  $G = S_n^{(k)} \oplus I_t$, where $k$ is the number of orbits of size $n$ and $t$ is the number of fixed points. 
\end{Lemma}

Finally, note that for $n>6$, according to \cite[Proposition~1.1]{lieb}, the two next possible sizes of an orbit in the action of $S_n$ are $2n$ and $\binom{n}{2}$. They correspond to the action of $S_n$  on the coset spaces $S_n:A_{n-1}$ and $S_n:S_{n-2}$. The proof of Theorem~\ref{main}, and earlier results in \cite{lieb}, show that it is a crucial fact  whether the action of $S_n$ has an orbit of size at least $\binom{n}{2}$ or not.

\section{Proof of the main result for $n>6$}\label{s:pr}

In some point of our proof below we apply results of \cite{devi} on quasiprimitive permutation groups. Recall that a permutation group is \emph{quasiprimitive} if each nontrivial normal subgroup is transitive. To determine the distinguishing number of an action $(G,X)$, rather then of partition into $d$ blocks, we will speak of $d$-labeling of $X$. In case $d=2$, the existence of a distinguishing $2$-labeling is equivalent to the existence of a subset $S\subset X$ preserved only by trivial permutations.

Many results in this area have been proved for symmetric groups $S_n$ satisfying $n>6$. In this section we prove Theorem~\ref{main} for $n>6$. For smaller $n$ there are some irregularities and that is why a few exceptions arise. The cases with $n\leq 6$ will be considered in the next section.
\bigskip

\textit{Proof of Theorem~\ref{main}}. ($n>6$).

\emph{Step 1}. First assume that $(S_n,X)$ has an orbit $Y$ of size $|Y|\geq \binom{n}{2}$. We prove that in this case $D(S_n,X)=2$.

Consider the action of $G$ on $Y$. This is equivalent to the action of $S_n$ on $S_n:H$, where $H$ is a subgroup of $S_n$. Since $|Y|\geq \binom{n}{2}$, 
$H \notin \{S_n,A_n,S_{n-1}, A_{n-1}\}$. Thus, this action is faithful and $(S_n,Y)=(S_n,S_n:H)$. We prove that $D(S_n,Y)=2$.

If $H$ is not a subgroup of $A_n$, then $A_n$ acts on $S_n:H$ transitively and therefore $(S_n,S_n:H)$ is quasiprimitive. 
Hence, the result is by 
\cite[Theorem 1, Theorem 2]{devi}. 
Indeed, it is enough to check that in \cite[Table 2]{devi} there is no group isomorphic with $S_n$. (It helps to use the fact that, since $n\geq 7$, the size $|Y|\geq \binom{7}{2}=21$, and so, there are only five groups to check.)

So assume that $H \leq A_n$. Then
each orbit in the action of $S_n$ on $S_n:H$ has an even number of cosets: half of them are the cosets $Hg$ with $g\in A_n$, and half of them are those with $g\in S_n\setminus A_n$. In particular, $Y = Y_1 \cup Y_2$ is the union of two disjoint sets $Y_1$ and $Y_2$ of cosets such that every element $g\in S_n\setminus A_n$ interchanges those sets: $Y_1g = Y_2$ and $Y_2g = Y_1$.

Hence, looking for the action of $A_n$ on $Y$, we see that this action has two orbits on $Y$, which are $Y_1$ and $Y_2$. Since $A_n$ is simple, the action of $A_n$ on each of these orbits is quasiprimitive. Using again \cite[Theorem 1, Theorem 2]{devi} (or \cite[Theorem 4.1]{gksim}) we get that $D(A_n,Y_1)=2$ with the exception of $A_8$ acting on 15 elements. 
Thus, with the above exception, there exist a subset $S$ of $Y_1$ such that no element $g\in A_n$ preserves it. Obviously, this set is not preserved by any element $g\in S_n\setminus A_n$, either (since they interchange $Y_1$ and $Y_2$).  Thus, $D(S_n,Y)=2$, as required. Using Lemma~\ref{f:d} yields  $D(S_n,X)=2$.

It remains to consider the case of $A_8$ acting on $15$ elements. Here $D(A_8,Y_1) = 3$. So we have a $3$-coloring $c_1:Y_1 \to \{1,2,3\}$ that is preserved by no permutation of $A_8$. Let $c_2: Y_2\to \{1,2,3\}$ be the corresponding $3$-coloring for $Y_2$ (in the parallel action of $A_8$ on both the orbits). Obviously, it would be enough to color only one of the orbits $Y_1$ or $Y_2$, but we make use of the second orbit to decrease the number of colors. Indeed, let $c: Y_1\cup Y_2 \to \{1,2\}$ be the $2$-coloring defined as follows: if $x\in Y_1$, then we put $c(x)=1$ if and only if $c_1(x)=1$; and if $x\in Y_2$, then we put $c(x)=1$ if and only if $c_2(x)=1$ or $c_2(x)=2$. Otherwise, we put $c(x)=2$. It is easily seen that no permutation of $A_8$ preserves this coloring, and no permutation interchanging $Y_1$ and $Y_2$ does so. 
Therefore, again, $D(S_n,X)=2$, which completes the proof of the Step~1.
\smallskip

\emph{Step~2}.
Now, assume that all orbits  of $(S_n,X)$ are of size $\leq \binom{n}{2}$. Then, by \cite{lieb}, the sizes of the orbits are in the set $\{1,2,n,2n\}$ corresponding to transitive actions of $S_n$ on $S_n:H$ with $H\in \{S_n,A_n,S_{n-1},A_{n-1}\}$. First, consider the case when the sizes are in the subset $\{1,n\}$. Then, by Lemma~\ref{f:n},  $G=S_n^{(k)} \oplus I_t$, for some $k\geq 1$ and $t\geq 0$.  We show that in this case  $D(G,X)=\left\lceil  \sqrt[k]{n}\; \right\rceil$. (This part of the proof works for $n\geq 3$; the assumption $n>s$ was used only in applying \cite{lieb}  above).

We need to show that for $d^k\geq n$, $(S_n,X)$ has a distinguishing $d$-labeling, and if $d^k < n$, then no such labeling exists for $G$.  
Suppose that $S_n$ acts on $X = X_1 \cup \ldots \cup X_{k} \cup Z$, where each $X_i$ is a copy of $\{1,2,\ldots,n\}$,  $Z$ is the set of fixed points (possibly empty), and the action of $S_n$ on all $X_i$ is the natural action of $S_n$ on $\{1,\ldots,n\}$.

Let  $\ell$ denote a $d$-labeling of $X$ with $\ell(i,j)$ denoting the label of $i$-th element in $X_j$ (the labeling of fixepoints is irrelevant). We define a labeling $\ell'$ of the set $\{1,2,\ldots,n\}$ with $k$-tuples of labels of $\ell$ by the formula $$\ell'(i) = (\ell(i,1),\ell(i,2),\ldots,\ell(i,k)).$$ Obviously, $\ell$ is a distinguishing $d$-labeling for $S_n^{(k)}\oplus I_t$ if and only if the $k$-tuples in $\ell'$ are pairwise distinct. Since the labels  $\ell'(i)$ are $k$-tuples of $d$ possible values, there are $d^k$ of them. So, if $d^k < n$, then the labeling $\ell'$ is not distinguishing, and consequently no labeling $\ell$ is distinguishing in this case. Otherwise, one may choose labels $\ell(i,j)$ so that $\ell$ is distinguishing. This proves that $D(S_n,X)=
\left\lceil  \sqrt[k]{n}\; \right\rceil$, as required.
 
(The labeling constructed above is essentially the same as in the proof of \cite[Theorem 4.2]{klav}. For the next part we need a more sophisticated argument).
 
Now, assume that we have an orbit of size  $2$ or $2n$.
The actions of $S_n$ on such orbits are equivalent to the actions of $S_n$ on $S_n/A_n$ and $S_n/A_{n-1}$, respectively. So, according to the argument given above, if we consider the action of $A_n$ on any of these orbits, then each such orbit consists of two orbits in the action of $A_n$ that are interchanged by the elements $g\in S_n\setminus A_n$. Thus, in the action of $A_n$ on $X$ we have only orbits of size $n$ or $1$. Since the only action of $A_n$ on $n$ points is the natural action of $A_n$, we infer, similarly as in Lemma~\ref{f:n}, that  $(A_n,X) = A_n^{k+2r} \oplus I_t$, where $t\geq 0$ is the number of fixed points in this action.

We use it to show that in this case $D(S_n,X)= \left\lceil  \sqrt[m]{n-1}\; \right\rceil$, where $m=2k+r$. First we show that $D(A_n,X)= \left\lceil  \sqrt[m]{n-1}\; \right\rceil$. The proof is similar to that above concerning the action $(S_n,X)$ with orbits of size $n$ and $1$ only.

We need to show that for $d^m\geq n-1$, $(A_n,X)$ has a distinguishing $d$-labeling, and if $d^m < n-1$, then no such labeling exists.

Again we assume that $X = X_1 \cup \ldots \cup X_{m} \cup Z$, where each $X_i$ is a copy of $\{1,2,\ldots,n\}$, $Z$ is the set of fixed points, and the action of $A_n$ on all $X_i$ is the natural action of $A_n$ on $\{1,\ldots,n\}$ By  $\ell$ we denote a $d$-labeling of $X$ with $\ell(i,j)$ denoting the label of $i$-th element in $X_j$, and define a labeling $\ell'$ of the set $\{1,2,\ldots,n\}$ by   $\ell'(i) = (\ell(i,1),\ell(i,2),\ldots,\ell(i,m))$. Now, $\ell$ is a distinguishing $d$-labeling for $A_n^{(m)}\oplus I_t$ if and only if  $\ell'$ is a distinguishing labeling of $A_n$ acting naturally on $\{1,2,\ldots,n\}$. The latter holds if and only if no more than two points have the same label (otherwise there is a nontrivial permutation in $A_n$ preserving the labeling). Since the labels $\ell'(i)$ are $m$-tuples of $d$ possible values, there are $d^m$ of them. Thus, if $d^m < n-1$, then either more then two points must have the same label or there are two pairs with the same label, and therefore the labeling is not distinguishing. Otherwise, one may choose labels $l(i,j)$ so that $\ell$ is distinguishing for $A_n^{(m)}\oplus I_t$.

Now we need to show that this labeling prevents any permutation in $(S_n,X)$ or it may be modified to have this property, using the same set of labels.
If there is an orbit of $(S_n,X)$ of size $2$ then all we need is to take care to attach different labels to the points in this orbit. Then no $g\in S_n\setminus A_n$ preserves this labeling. If we have no orbit of size $2$, then by assumption there is an orbit of size $2n$, $Y=Y_1\cup Y_2$, where $Y_1$ and $Y_2$ are orbits of $A_n$. If the labeling $\ell$ of $Y_1$ and $Y_2$  is  different (meaning there is a label occurring in $Y_1$ and $Y_2$ different numbers of times), then it protects the possibility of exchanging $Y_1$ and $Y_2$, and we are done: no $g\in S_n\setminus A_n$ preserves such labeling. Since we have some freedom in choosing labels $\ell(i,j)$ (under condition $d^m \leq n-1$), we usually are able to choose them so that labeling of $Y_1$ and $Y_2$ form different multisets, which yields a required labeling. The only case when it may be impossible is when $n=d^m$ or $n=d^m+1$.

In the first case, when all $d^m$ sequences are used in labeling $\ell'$, then the multisets of labels of $\ell$ for all orbits $X_j$ are the same. A solution is to modify $\ell'$ so that the label $\ell(1,j_1)$ for $Y_1$ is changed, while the label $\ell(1,j_1)$ for $Y_2$ remains the same. Then there is a pair of identical sequences in $\ell'$, but since the remaining are pairwise distinct, the labeling still protects any permutation in $A_n$. On the other hand, the derived labeling $\ell$ for $Y_1$ and $Y_2$ is now different, which is as required.

In the second case, $n=d^m+1$, when all $d^m$ sequences are used in labeling $\ell'$, precisely one must be used twice. Similarly, as above we may change one of this two sequences, so that the labeling $\ell$ for $Y_1$ and $Y_2$ is different. Since such a change results only in the change of the sequence appearing in $\ell'$ twice, the obtained labeling is again, as required. 
 The proof is completed.
\hfill$\Box$\bigskip

\section{Structure of permutation groups and graphs with $D>2$} \label{s:re}

We describe the permutation groups $(G,X)$ abstractly isomorphic to $S_n$ with the distinguishing number $D(G,X)>2$. We use the terminology  introduced in Section~\ref{s:ss}. In all cases, when isomorphisms or homomorphisms involved in notation are natural, clear from the context and unique up to permutations of points, we omit their specifications.

It should be clear that for every $k> 0$ and $r\geq 0$, as well as for every pair $k=0$ and $r>0$ there exists an action of $S_n$ with the distinguishing number as described in Theorem~\ref{main}. 
It is enough to take a set partitioned into suitable orbits and to define a parallel action of $S_n$ on all orbits equivalent to actions on $S_n : S_{n-1}$, $S_n: A_{n-1}$, and $S_n: A_{n}$, respectively. 
We can describe these actions directly as sets of permutations.
Clearly, the corresponding permutation groups are of the following form 
$$G=((S_n^{(k)} \oplus_\phi (S_n,2n)^{(r)})[A_n^{(k+2r)}] \oplus_\psi   S_2^{(s)}) \oplus I_t,$$ where $(S_n,2n)$ denotes the permutation group corresponding to the unique action of $S_n$ on $10$ points (which is the action on $ S_n:A_{n-1}$). The formula for the distinguishing number in Theorem~\ref{main} works 
for $3\leq n \leq 6$, as well.

Now, we describe the actions of $S_n$ that have an exceptional distinguishing number. All they occur for some $n\leq 6$.
First, there are two transitive actions of $S_4$ on $6$ points called in the literature $S_4(6c)$ and $S_4(6d)$. 
Using GAP (or another combinatorial computation system) one can check that the distinguishing number for both
$S_4(6c)$ and $S_4(6d)$ is equal to $3$.

Another action of $S_4$ on $6$ points is $(S_4,S_4 : K_4)$. where $K_4$ is the Klein 4-group. Since $S_4/K_4 \cong S_3$, this is a non-faithful action of $S_4$ equivalent to a faithful action of $S_3$ on $6$ points. This may be used to obtain an intransitive  action of $S_4$ on two orbits of $4$ and $6$ elements.  The corresponding permutation group is $S_4[K_4]\oplus_\phi S_3$, where $\phi$ is the isomorphism of $S_4/K_4 \to S_3$, unique up to permutation automorphism of $S_3$. Using GAP we compute that $D(S_4[K_4]\oplus_\phi S_3)=3$. Moreover, here (as well as in the previous cases) one can add an arbitrary number $t$ of fixed points, and we have still $D((S_4[K_4]\oplus_\phi S_3)\oplus I_t)=3$ for $t>0$.

The projective linear group $PGL(2,5)$ is isomorphic abstractly to $S_5$. We compute with GAP that $D(PGL(2,5)) =4$ (while $D(S_5)=5$). Also we have $D(PGL(2,5)[PSL(2,5)] \oplus_\phi S_2^{(s)})=3$, $s>0$, where  $\phi: PGL(2,5)/PSL(2,5) \to  S_2^{(s)}$ is again unique, as above.

A well-known exception for $n=5$ is the automorphism group $Aut(\mathcal P)$ of the Petersen graph $\mathcal P$, which is known to be isomorphic to $S_5$. We have $D(Aut({\mathcal P}), 10) = 3$.

For $n=6$, there is an exceptional action of $S_6$ on $10$ points $(S_6,10)$. For this action $D(S_6,10)=3$ and $D((S_6,10)[A_6,10] \oplus_\phi S_2^{(s)})=3$ for any $s>0$, with $\phi : (S_6,10)/(A_6,10) \to S_2^{(s)}$.

Finally, we have  $D(S_6^{(2)},12)=3$, which is included in (1) of Theorem~\ref{main}. Yet, $S_6$ has also a nonpermutation automorphism $\psi$ and the corresponding parallel sum $S_6\oplus_\psi S_6$ is not permutation isomorphic with
$S_6^{(2)}$. Still, as we check, we have $D(S_6\oplus_\psi S_6,12)=3$ and 
$D((S_6\oplus_\psi S_6) \oplus_\phi S_2^{(s)})=3$ for any 
$s>0$, where $\phi : (S_6\oplus_\psi S_6)/(A_6\oplus_\psi A_6) \to S_2^{(s)}$.

These findings are summarized in Table~1. In the table we also add possible fixed points $I_t$, and apply the convention that for $t=0$ the summand $I_t$ is simply omitted. An analogous convention is applied for summands $S_2^{(s)}$.

\begin{table}[ht!]
\begin{center}
\caption{Permutation groups isomorphic to $S_n$ with an exceptional distinguishing number $D$.}
\label{t:1}
 
\small
\begin{tabular}{ |l|l|c|l| }
\hline
$S_n$ & Group & $D$ & parameters  \\ \hline  $S_4$ & $S_4(6c)\oplus I_t$ & 3 & $t \geq 0$   \\  
$S_4$ & $S_4(6d)\oplus I_t$ & 3 & $t \geq 0$ \\  
$S_4$ & $(S_4[K_4] \oplus_\phi S_3) \oplus I_t$ & 3 & $t \geq 0$ \\  
$S_5$ & $PGL(2,5) \oplus I_t$ & 4 & $t \geq 0$ \\
$S_5$ & $(PGL(2,5)[PSL(2,5)] \oplus_\phi S_2^{(s)}) \oplus I_t$ & 3 & $s>0$, $t \geq 0$ \\
$S_5$ & $Aut(\mathcal P) \oplus I_t$ & 3 & $t \geq 0$ \\
$S_6$ & $((S_6,10)[(A_6,10)] \oplus_\phi S_2^{(k)}) \oplus I_t$ & 3 & $s \geq 0$, $t \geq 0$ \\
$S_6$ & $((S_6 \oplus_\phi S_6)[A_6 \oplus_\phi A_6] \oplus_\phi S_2^{(s)}) \oplus I_t $ & 3 & $s \geq 0$, $t \geq 0$ \\ \hline
 \end{tabular}
\end{center}
\end{table}

The permutation groups in Table~1 are the only exceptions as the following result shows.

\begin{Theorem}\label{exc} 
Let $G$ be a permutation group on a set $X$ abstractly isomorphic to $S_n$, for some  $n\geq 3$. If $D(G,X)\neq 2$ and $G$ is not an exception listed in Table~$1$, then 
$$G =  ((S_n^{(k)} \oplus_\phi (S_n,2n)^{(r)})[A_n^{(k+2r)}] \oplus_\psi   S_2^{(s)}) \oplus I_t,$$

\noindent for some $k,r,s,t\geq 0$, $k+r>0$, and 
      
  $$D(G,X)=\left\{
 \begin{array}{ll}
\left\lceil  \sqrt[k]{n}\; \right\rceil, & \textrm{if } r+s = 0,  \\
\lceil  \sqrt[k+2r]{n-1}\; \rceil, & \textrm{otherwise.}
 \end{array}\right.$$ 
\end{Theorem}

\begin{proof}
The proof for $n>6$ has been given in Section~\ref{s:pr}. The remaining part, for $n\leq 6$, is by considering all possible cases and using computational machinery of GAP or another similar system. For each $n\leq 6$ our strategy is the same and will be illustrated on the example of $n=4$.

First, generally, we look for permutation groups isomorphic with $S_n$ using Theorem~\ref{intr}. We establish all possible constituents, which according to Lemma~\ref{f:a}, are $S_n/H$, where $H$ is a normal subgroup of $S_n$. Since, for $n\neq 4$, $A_n$ is the only nontrivial normal subgroup of $S_n$, the possible constituents are $S_n,S_2$ and $S_1$ (a fixed point), and in case of $n=4$ we have one more possible constituent $S_3\cong S_4/K_4$.

Now, we establish the possible sizes of orbits for a faithful action of $S_n$. These are sizes of coset spaces $S_n:H$ for any subgroup $H<S_n$. For $n=4$ these are all divisors of $24$.

Using Lemma~\ref{f:f} we know that we have a faithful action of $S_n$ on one of the orbits of size $\geq n$. For $n=4$ we check that there are unique transitive actions on $24, 12$ and $8$ points, and $D(S_4,24)=2$, $D(S_4,12)=2$, and $D(S_4,8)=2$. By Lemma~\ref{f:d}, the distinguishing number of a permutation group containing any orbit of this type is $2$, so there is nothing more to check in this case.

For $6$ points we have two actions of $S_4$, described at the beginning of this section, with (the distinguishing number) $D=3$. This requires to check other possible orbits with the constituent $S_4$ (here we may restrict to orbits of size $\leq 6$) and with constituents $S_3$ and $S_2$. We check that each of four possible permutation groups with the action of $S_4$ on two orbits of size $6$ has  $D=2$. So, similarly as above, using Lemma~\ref{f:d}, we do not need to consider other possible orbits. Analogously the constituents $S_4$ on $6$ points and $S_3$ on $6$ or $3$ points yields four possible groups with $D=2$. There remains to check a finite, relatively small number of further cases showing that two first actions of $S_4$ on $6$ points are the only exceptions with $D>2$.

Using the same strategy for other $n\leq 6$ we obtain that the only exceptions not covered by $(3)$ of Theorem~\ref{main} are those listed in Table~1.
\end{proof}

An extended version of Theorem~\ref{main2} for graphs is the following:

\begin{Theorem} \label{exc2}
Let $\G$ be a graph with $Aut(\G)\cong  S_n$. If $D(\G)\neq 2$, and $\G$ is not the Petersen graph or its complement, nor any extension  of these graphs with fixed points, then  
\begin{center}
$Aut(\G)= S_n^{(k)}  \oplus I_t$  and  $D(\G)=\left\lceil  \sqrt[k]{n} \; \right\rceil$, \end{center}for some $k>0$ and  $t\geq 0$.
\end{Theorem}\begin{proof}
Liebeck \cite[Proposition~1.2]{lieb} has proved that under conditions of the theorem, for $n>6$, if the sizes of all orbits in $Aut(\G)$ are less then $\binom{n}{2}$, then all orbits have size $1$ or $n$. This combined with our Theorem~\ref{exc} yields the result for $n>6$.

For $n\leq 6$, it is enough to check that none of the exceptions listed in Theorem~\ref{exc}, other than $Aut(\mathcal P)$, is the automorphism group of a graph. This can be easily done with GAP.
\end{proof}

Again, for every $n\geq 3$ and $k\geq 1$ there exist graphs fulfilling the conditions in Theorem~\ref{exc2}. Such graphs are described in  \cite{lieb}, where a number of necessary conditions (mentioned in our introductory section) are given for a graph to have the automorphism group isomorphic with $S_n$ and with orbits of sizes $1$ and $n$ only. In \cite{lieb}, Liebeck also posed a conjecture  that these conditions are sufficient, but this conjecture remains open.

\end{document}